\documentclass[10pt]{amsart}

\usepackage{amsfonts,amsmath,latexsym,amssymb,verbatim,amsbsy,amsthm}
\usepackage{dsfont,bm}

\usepackage[top=1in, bottom=1in, left=1in, right=1in]{geometry}

\usepackage[dvipsnames]{xcolor}

\usepackage[colorlinks=true, pdfstartview=FitV, linkcolor=RoyalBlue,citecolor=ForestGreen, urlcolor=blue]{hyperref}

\theoremstyle{plain}
\newtheorem{THEOREM}{Theorem}[section]

\newtheorem{theorem}[THEOREM]{Theorem}
\newtheorem{corollary}[THEOREM]{Corollary}
\newtheorem{lemma}[THEOREM]{Lemma}
\newtheorem{proposition}[THEOREM]{Proposition}

\theoremstyle{definition}

\newtheorem{definition}[THEOREM]{Definition}

\theoremstyle{remark}

\newtheorem{remark}[THEOREM]{Remark}
\newtheorem{claim}[THEOREM]{Claim}

\newcommand{\thm}[1]{Theorem~\ref{#1}}
\newcommand{\lem}[1]{Lemma~\ref{#1}}

\newcommand{\prop}[1]{Proposition~\ref{#1}}

\newcommand{\sect}[1]{Section~\ref{#1}}

\def \d {\delta}
\def \e {\varepsilon}

\def \k {\kappa}
\def \l {\lambda}
\def \n {\nabla}
\def \s {\sigma}

\def \D {\Delta}

\def \O {\Omega}

\def \cA {\mathcal{A}}

\def \cD {\mathcal{D}}
\def \cE {\mathcal{E}}

\def \cG {\mathcal{G}}
\def \cH {\mathcal{H}}
\def \cI {\mathcal{I}}

\def \cR {\mathcal{R}}

\newcommand{\N}{\ensuremath{\mathbb{N}}}   
\newcommand{\R}{\ensuremath{\mathbb{R}}}   
\newcommand{\T}{\ensuremath{\mathbb{T}}}   
\newcommand{\I}{\ensuremath{\mathbb{I}}}   

\def \loc {\mathrm{loc}}

\def \one {{\mathds{1}}}

\def \Lip {\mathrm{Lip}}

\newcommand{\ave}[1]{\langle #1 \rangle}

\DeclareMathOperator{\diam}{diam} %
\def \lan {\langle}
\def \ran {\rangle}

\def \p {\partial}

\def \ss {\subset}

\def \GL {Gr\"onwall's Lemma}
\def \HI {H\"older inequality}

\def \CK{Csisz\'ar-Kullback inequality}

\renewcommand{\geq}{\geqslant}

\renewcommand{\leq}{\leqslant}

\def \dx  {\, \mbox{d}x}
\def \dxi  {\, \mbox{d}\xi}

\def \dy  {\, \mbox{d}y}

\def \dr  {\, \mbox{d}r}

\def \dmu  {\, \mbox{d}\mu}
\def \deta  {\, \mbox{d}\eta}

\def \dv  {\, \mbox{d} v}

\def \ddt  {\frac{\mbox{d\,\,}}{\mbox{d}t}}

\def \dW {\dot{W}}
\def \rmin{\underline{\rho}}

\def \domain {{\O^n \times \R^n}}
\def \ufilt {u_{\phi,\rho}}
\def \uf {u_{\phi,\rho}}
\def \tuf {\tilde{u}_{\phi,\tilde{\rho}}}

\def \uFavre {u_{\mathrm{F}}}
\def \uF {u_{\mathrm{F}}}

\begin{document}

\title{Global hypocoercivity of kinetic Fokker-Planck-Alignment equations}

\author{Roman Shvydkoy}

\address{851 S Morgan St, M/C 249, Department of Mathematics, Statistics and Computer Science, University of Illinois at Chicago, Chicago, IL 60607}

\email{shvydkoy@uic.edu}


\date{\today}


\thanks{\textbf{Acknowledgment.}  
	This work was  supported in part by NSF
	grants DMS-1813351 and DMS-2107956.}

\begin{abstract}
In this note we establish hypocoercivity and exponential relaxation to the Maxwellian  for a class of kinetic Fokker-Planck-Alignment equations arising in the studies of collective behavior.  Unlike previously known results in this direction that focus on  convergence near Maxwellian, our result is global for hydrodynamically dense flocks, which has several consequences. In particular, if communication  is long-range, the convergence is unconditional. If communication is local then all nearly aligned flocks quantified by smallness of the Fisher information relax to the Maxwellian. In the latter case the class of initial data is stable under the vanishing noise limit, i.e.\ it reduces to a non-trivial and natural class of traveling wave solutions to the noiseless Vlasov-Alignment equation.

The main novelty in our approach is the adaptation of a mollified Favre filtration of the macroscopic momentum into the communication protocol. Such filtration has been used previously in large eddy simulations of compressible turbulence and its new variant appeared in the proof of the Onsager conjecture for inhomogeneous Navier-Stokes system. A rigorous treatment of well-posedness  for smooth solutions  is provided. Lastly, we prove that in the limit of strong noise and local alignment  solutions to the Fokker-Planck-Alignment equation Maxwellialize to solutions of the macroscopic hydrodynamic system with the isothermal pressure. 

\end{abstract}

\maketitle

\section{Background and Motivation}

One of the most fundamental problems that arise in studies of collective behavior of large systems is to understand emergence of global phenomena from purely local interactions. The Hegselmann-Krause model of opinion dynamics \cite{HKmodel} or Vicsek model of swarming \cite{VCBCS1995} provide examples of such phenomena and are well-studied in the applied literature. In the context of alignment dynamics a class of environmental averaging models, such as Cucker-Smale \cite{CS2007a,CS2007b}, Motsch-Tadmor \cite{MT2011} and their topological counterparts \cite{STtopo,Haskovec2013,MMP2020} provide analytical framework for studying emergence in the sense of convergence to a common state $v_i \to \bar{v}$, a ``consensus", see \cite{Sbook,VZ2012,MT2014,MMPZ-survey} for detailed surveys. 

Let us give a brief overview of the problem and most recent known developments. Every alignment model consist of two core components. Those are some averaging protocol
\[
v=(v_1,\dots, v_N) \in \R^{nN} \to ( \ave{v}_1,\dots, \ave{v}_N)\in \R^{nN},
\]
where each bracket $\lan \cdot \ran_i$ encapsulates probing of the environment $\O^n$ in a  neighborhood of agent $x_i$, 
and a communication strength function $\k_i(x)$.  Both components may depend on positions of all agents $x = (x_1,\dots,x_N)$. A general alignment system is then given by
\begin{equation}\label{e:ave}
\begin{split}
\dot{x}_i & = v_i \\
\dot{v}_i & = \k_i(x) ( \ave{v}_i - v_i ).
\end{split}
\end{equation}
The classical Cucker-Smale model  is a well-studied example of such a system given by
\begin{equation}\label{e:CS}
\dot{v}_i  =\sum_{j=1}^N m_j \phi(x_i-x_j)(v_j - v_i ),
\end{equation}
where $\phi$ is a radial communication kernel and $m_j$'s represent  communication weights of agents. In this case the strength $\k_i(x) =  \sum_{j=1}^N m_j \phi(x_i-x_j)$ is based on metric proximity of the crowd around. If $\k_i(x) = 1$, one obtains the Motsch-Tadmor pure averaging model
\begin{equation}\label{e:MT}
\dot{v}_i  =  \ave{v}_i - v_i, \qquad  \ave{v}_i = \frac{\sum_{j=1}^N m_j \phi(x_i-x_j) v_j}{\sum_{k=1}^N m_k \phi(x_i-x_k)}.
\end{equation}
In both cases, exponential alignment is achieved under long range fat-tail condition $\int_0^\infty \phi(r) \dr =\infty$, see \cite{CS2007a,HL2009,MT2014}. 

Under local averaging rules it is generally impossible to achieve alignment by an obvious counterexample: if $\O^n = \R^n$ one simply sends two agents in opposite directions, and if $\O^n $ is periodic $\T^n$ one can send two agents along perpendicular geodesics with relatively rational velocities so that the agents never approach each other closer than a communication range $r_0 \ll 1$ resulting in a so called {\em locked state}. Both counterexamples can be ruled out assuming graph-connectivity of the flock at scale $r_0$. If the flock is nearly aligned and initially connected such an assumption will propagate in time resulting in exponential alignment, see \cite{Sbook,MPT2019} and references therein.

In what follows we will restrict ourselves to the periodic environment $\O^n = \ell \T^n$ to focus more on the dynamics in the bulk of a flock, and to avoid technical issues related to confinement, see however \cite{ShuT2019,Villani}. We will also be interested in large systems, $N\to \infty$, which support kinetic description  via a mean-field limit, see \cite{HL2009},
\begin{equation}\label{}
f_t + v \cdot \n_x f = \n_v( \k(x) (v - \ave{u}) f).
\end{equation}
Here, $\ave{u}$ is the corresponding averaging operation of the macroscopic velocity field, and $\k(x)$ is the  limiting communication strength. For example, in the CS case, $\k(x) = \phi \ast \rho: = \rho_\phi$ and  $\ave{u} = \frac{(u \rho)_\phi}{\rho_\phi}$, where $\rho$ and $u\rho$ are the macroscopic density and momentum,
\[
\rho(x) = \int_{\R^n} f(x,v) \dv, \qquad u\rho(x) = \int_{\R^n} v f(x,v) \dv.
\]

The issue of locked states becomes the primary obstacle for global alignment on the periodic domain. Such states are highly unlikely as they form a negligible set of data. So, it is conceivable that a deterministic approach can be successful in proving emergence for generic initial conditions, however this has only been done in 1D, \cite{DScorr}, where dimensional restrictions are severe. A more natural approach is to disrupt locked states by incorporating a small properly scaled noise
\begin{equation}\label{ }
\dot{v}_i  = \k_i(x) ( \ave{v}_i - v_i ) + \sqrt{2\s \k_i(x)} \dW_i,
\end{equation}
where $W_i$'s are independent Brownian motions in $\R^n$. The mean-field limit of solutions satisfies a Fokker-Planck-Alignment equation  (although such limit has be verified for each particular model, see for example \cite{Rosello2020} and references therein)
\begin{equation}\label{e:FPAgen}
f^\s_t + v \cdot \n_x f^\s = \s  \k(x) \D_v f^\s + \n_v( \k(x) (v - \ave{u^\s}) f^\s).
\end{equation}

Since any noise disrupts the occurance of locked states we anticipate that they would play no role in the long time dynamics of \eqref{e:FPAgen}. So, the expected behavior as $t \to \infty$ would be the same as for the linear Fokker-Planck equation which is a relaxation to the global Maxwellian
\begin{equation}\label{e:relax}
f^\s \to \mu_{\s,\bar{u},M} = \frac{M}{|\O^n|(2\pi \s)^{n/2}} e^{- \frac{|v - \bar{u}|^2}{2\s}},
\end{equation}
where $\bar{u}$ is some constant velocity vector, and $M$ is the total mass.  If such a convergence holds true, then the alignment of the original system can be recovered in the limit of vanishing noise $\s \to 0$:
\begin{equation}\label{e:ultimate}
\lim_{\s \to 0} \lim_{t \to \infty} f^\s(t)  = \frac{M}{|\O^n|} \d_{v = \bar{u}} \otimes \dx.
\end{equation}

This program has seen partial success. In \cite{DFT2010} Duan, Fornasier, and Toscani proved relaxation \eqref{e:relax} in the Cucker-Smale case for the near-Maxwellian initial data $f_0$ in the strong Sobolev metric,
\begin{equation}\label{e:nearM}
\begin{split}
f = \mu_{1,\bar{u},M} + g \sqrt{\mu_{1,\bar{u},M}} , \qquad \| g_0\|_{H^k(\domain)} \leq \e,
\end{split}
\end{equation}
 for some small $\e>0$. Although in this case the alignment term $\k(x) \ave{u} = (u \rho)_\phi$ is smooth, which avoids issues with well-posedness, the system does not have a globally decaying Lyapunov function -- entropy. A similar result was alluded to in \cite{DFT2010} for the Motsch-Tadmor case, which also suffers from the lack of entropy. More recently, Choi \cite{Choi2016} demonstrated the limit \eqref{e:relax} for purely local Motsch-Tadmor model where $\k = 1$ and $\phi = \d_0$, i.e. $\k \ave{u} = u$. The limit as $\phi \to \d_0$ was justified in \cite{KMT2014}. In this case the equation has an entropy, but the fully nonlinear nature of the alignment force requires delicate energy estimates. The result is proved for the same near-Maxwellian data \eqref{e:nearM} and convergence holds exponentially fast on the torus.

Both of these results are largely inspired by techniques developed for collisional models, \cite{Duan2007,Guo2002}, where the perturbative analysis is adapted to dealing with the particular structure of nonlinear averaging.  For alignment models, however, such issues seem to be more of a technical origin rather than related to any specific phenomenological obstruction mentioned earlier.   So, our goal in this present work is to fulfill the need for a global relaxation result departing from the near-Maxwellian settings \eqref{e:nearM} and relying instead on the natural characteristics of the flock such as connectivity or communication.

Let us assume that $\phi \in C^{\infty}(\O^n)$ is a convolution type communication kernel on the periodic domain $\O^n = \ell\T^n$ satisfying
\begin{equation}\label{e:kernel}
\int_{\O^n} \phi(x) \dx = 1, \qquad \phi(x) \geq c_0 \one_{|x|<r_0}.
\end{equation}
We call $\phi$ {\em global} if $\phi$ is bounded from below on the domain, i.e. $r_0 = \diam \O^n$. 

To state the particular kinetic model we will be interested in, we define the following 
 density-weighted filtration of the macroscopic field $u$:
 \begin{equation}\label{e:ufilt}
\ufilt = \left( \uF \right)_\phi, \qquad \uF = \frac{(u \rho)_\phi}{\rho_\phi}.
\end{equation}
The expression for $\uF$ is exactly the macroscopic analogue of the Motsch-Tadmor averaging. In fact, in the compressible turbulence this is known as the Favre filtration, see \cite{Favre},  used for large eddy simulations. One of the notable properties of the Favre filtration is that the mollified density satisfies the continuity equation relative to $\uF$,
\[
\p_t \rho_\phi = \n\cdot (\uF \rho_\phi),
\]
which makes it more accessible numerically. The extra mollification that defines our averaging protocol \eqref{e:ufilt} makes it suitable for a number of applications. First, it was implemented in the proof of the energy conservation for solutions of inhomogeneous Navier-Stokes system in the Onsager-critical spaces, see \cite{LS2016} (although defined in terms of Littlewood-Paley projections). In the context of alignment models it was instrumental in extending Figalli and Kang's hydrodynamic limit result, \cite{FK2019}, to flocks with finite support, see \cite{Sbook}. The mean-field limit of the discrete system 
\begin{equation}\label{e:MTave}
\dot{v}_i  =  \ave{v}_i - v_i, \qquad  \ave{v}_i = \int_{\O^n} \phi(x_i - y) \frac{\sum_{j=1}^N m_j \phi(y-x_j) v_j}{\sum_{k=1}^N m_k \phi(y-x_k)} \dy
\end{equation}
to the corresponding Vlasov-Alignment model
\begin{equation}\label{e:VA}
\p_t f + v\cdot \n_x f =  \n_v((v - u_{\phi,\rho})f )
\end{equation}
was also justified in \cite{Sbook}.  A notable distinction between \eqref{e:MTave} and the Motsch-Tadmor model \eqref{e:MT} is that the averaged version preserves the total momentum, and hence, the limiting velocity $\bar{u}$ is determined directly from the  initial condition. 

With the added noise we arrive at the following Fokker-Planck-Alignment equation which will be the main object of our study 
\begin{equation}\label{e:FPA}
\p_t f + v\cdot \n_x f = \s \D_v f + \n_v((v - u_{\phi,\rho})f ).
\end{equation}
Thanks to the new averaging protocol the model possesses a number of remarkable properties including global hypocoercivity, which enables us to establish a general relaxation result. We describe it in the next section.

\section{Main results}
The Fokker-Planck-Alignment equation \eqref{e:FPA}, FPA for short, obeys two conservation laws -- 
 mass and momentum 
\[
\bar{u} = \frac1M \int_\domain v f(x,v) \dx \dv, \qquad M = \int_{\domain} f(x,v) \dx \dv.
\]
Thus, the macroscopic limiting parameters of the model are determined by the initial values.  The model also possesses a Lyapunov function -- relative entropy which we address in detail in \sect{s:ee}.  In addition, the model is locally well-posed in  weighted Sobolev classes
\[
H^k_s(\domain) = \left\{ f :  \| f\|_{H^k_s}^2 = \int_\domain (1+|v|^2)^{s/2} |\p^k_{x,v} f |^2 \dx \dv <\infty \right\},
\]
for hydrodynamically dense flocks, $\rho_\phi >0$, which we will prove in the Appendix. The latter is necessary to define the filtration $\ufilt$, and in fact, implies that $\ufilt \in C^\infty$. Moreover, the solutions will not blowup as long as $\rho_\phi$ remains positive. In particular, if $\phi$ is global, then $\rho_\phi \geq M \min \phi$, and in the case the FPA is globally well-posed.

We will need a more quantitative definition of the hydrodynamic density.

\begin{definition}\label{d:hc}
We say that the flock is {\em uniformly hydrodynamically dense} at a scale $r>0$ if there exists an adimensional $\d >0$ such that 
\begin{equation}\label{e:hc}
\inf_{t>0,\, x\in \O^n} \frac{1}{M} \int_{|x-\xi| <r} \rho(\xi,t) \dxi \geq \d.
\end{equation}
\end{definition}

Note that if a flock is dense at scale $r'$, it is dense at any larger scale $r''>r'$, and every flock is trivially dense at the global scale $r = \diam \O^n$. It is also clear that every part of a dense flock can be connected by a graph with legs of size $<r$. In fact, dense flocks are also automatically chain connected at scale $r$ in the sense defined in \cite{MPT2019}.  If $r=r_0$, where $r_0$ is the communication range \eqref{e:kernel}, then clearly 
\begin{equation}\label{e:rlowfromdens}
\rho_\phi \geq c(\d,c_0,r_0) >0. 
\end{equation}

Our main result states that all that is required for solutions to relax is hydrodynamic density at a scale smaller than communication range $r_0$.

\begin{theorem}\label{t:main}
Suppose $f \in H^k_s $ is a classical global solution to \eqref{e:FPA}, $\s <\s_0$, which is uniformly hydrodynamically dense at a scale $r<r_0$. Then $f$  relaxes to the corresponding Maxwellian at an exponential rate
\begin{equation}\label{e:relaxexp}
\|f(t) - \mu_{\s,\bar{u},M}\|_{L^1(\domain)} \leq c_1 M e^{- c_2 \s^{1/2} t},
\end{equation}
for some $c_1,c_2>0$ depending only on the parameters $\d,r,r_0,\s_0$, and $\O^n, \phi$.
\end{theorem}

As we noted earlier both hydrodynamic density and global well-posedness follow automatically for global communication kernels. So, in this case we obtain an unconditional result.

\begin{corollary}\label{}
If communication $\phi$ is global on $\O^n$, then the Fokker-Planck-Alignment equation \eqref{e:FPA} is globally well-posed in $H^k_s(\domain)$ and any solution satisfies \eqref{e:relaxexp}.
\end{corollary}

Next, we isolate a class of solutions that remain hydrodynamically dense if initially so. Let us consider the full Fisher information associated with a distribution $f$:
\begin{equation}\label{}
\cI(f) = \int_\domain \frac{\left| \s \n_v f+ (v - \bar{u}) f \right|^2 + \s |\n_x f|^2}{f} \dv \dx.
\end{equation}

\begin{theorem}\label{t:sid}
There is are constants $\e, c_1, c_2  >0$ depending only on $\O^n$ and $\phi$ such that
if $f_0 \in  H^k_s$ satisfies
\begin{equation}\label{e:Ismall}
\cI(f_0) \leq \e \s M,
\end{equation}
then there exists a unique global solution $f\in L^\infty_\loc(\R^+; H^k_s)$ to \eqref{e:FPA} with $f(0) = f_0$. Such solution will relax to the corresponding Maxwellian at an exponential rate \eqref{e:relaxexp}.
\end{theorem}

 Let us discuss the meaning of the smallness condition \eqref{e:Ismall}. By the classical log-Sobolev inequality the Fisher information dominates, and scales like, the relative entropy
 \begin{equation}\label{e:logSobI0}
\cI(f_0) \geq  \s \int_{\domain} f_0 \log \frac{f_0}{\mu_{\s,\bar{u},M}} \dv \dx,
\end{equation}
which in turn by the \CK\ dominates, and scales like, the $L^1$-norm of the difference
\[
\geq  \frac{c \s}{M} \| f_0 - \mu_{\s,\bar{u},M} \|_{L^1}^2.
\]
So, our condition  \eqref{e:Ismall} expresses a weaker form of proximity to Maxwellian than \eqref{e:nearM}. Additionally, in the limit as $\s \to 0$, the condition \eqref{e:Ismall}  does not degenerate into $f_0 =  \frac{M}{|\O^n|} \d_{v = \bar{u}} \otimes \dx$. In fact, for the ansatz $f_0 = \rho_0(x) \mu_{\s,\bar{u},M}(v)$,  \eqref{e:Ismall} translates into a $\s$-independent inequality
\[
\int_{\O^n} |\n_x \sqrt{\rho_0}|^2 \dx \leq \e,
\]
which expresses a measure of flatness of the initial density. Thus, in the limit as $\s\to 0$, condition \eqref{e:Ismall} still holds for a non-trivial class of data $f_0 = \rho_0(x) \d_{v = \bar{u}} $, which in fact produce the natural traveling wave solutions to the noiseless Vlasov equation \eqref{e:VA}, $f = \rho_0(x - t \bar{u}) \d_{v = \bar{u}}$, the so called flocking states, see \cite{ST2}. In fact, these would be solutions  for any profile $\rho_0$, which suggests that there might be a room for improvement in condition \eqref{e:Ismall}.

The proof of \thm{t:main} and \ref{t:sid} is focused on establishing the global hypocoercivity property of the FPA in the entropic settings. The general methodology follows Villani's treatment of the linear Fokker-Planck equation \cite{Villani} where the equation for the distribution $h = f/\mu$ is represented as a sum of the degenerate dissipative, transport and, in our case, alignment components
\[
\p_t h = - \s A^* A h - B h + A^*(\ufilt h),
\]
see \sect{s:hypo} for notation.  The main idea is to modify the Fisher information $\cI$ to include a properly scaled cross-product term 
\[
\cI_{xv}(h) = \s^{3/2} \int_\domain \frac{\n_x h \cdot \n_v h}{h} \dmu,
\]
and let the transport $B$ compensate for the lack of dissipation in $A^*A$.
The new information $\tilde{\cI} = \cI + \cI_{xv} \sim \cI$ in combination with the  relative entropy $\cH$ forms a global Lyapunov function satisfying
\begin{equation}\label{e:HIgronwall}
\ddt \left[ c_1 \cH+ \tilde{ \cI} \right] \leq - c_2 \s^{1/2}  \left[ c_1 \cH+ \tilde{ \cI} \right].
\end{equation}
An application of  the \CK \ yields the desired result.

 It is in the proof of \eqref{e:HIgronwall} where the special choice of filtration $\ufilt$ starts to play a crucial role. Specifically, the non-linear alignment term $A^*(\ufilt h)$ interacts in a particular way with the Fokker-Planck component and the entropy depletion terms to produce  cancellations necessary for the global rather than near-Maxwellian coercivity. We will provide full details of this computation although the reader will notice that the Fokker-Planck part is a more explicit form of Villani's abstract argument \cite{Villani}, which also appeared in more general Riemannian settings in \cite{Calogero2012}.

Finally, in \sect{s:hydro} we provide a rigorous derivation of the corresponding isentropic Euler-Alignment system 
\begin{equation}\label{e:macrointro}
\begin{split}
\rho_t  + \n \cdot (u \rho) & = 0\\
(\rho u)_t + \n \cdot (\rho u \otimes u) +  \n \rho &=  \rho( \ufilt - u),
\end{split}
\end{equation}
as the hydrodynamic limit of solutions to the strong noise/local alignment kinetic FPA  \eqref{e:FPAe}. The classical Cucker-Smale case was settled previously by Karper, Mellet, and Trivisa  in \cite{KMT2015}.  Here we take a more direct approach by  accessing  the full kinetic relative entropy rather than the macroscopic one, which enables us to provide a more economical argument.

\section{Entropy and Energy}\label{s:ee}

Let us make a few assumptions that will simplify to the notation. By the Galilean invariance 
\[
f(t,x,v) \to f(t,x+tV, v+V),
\]
we can assume that the total momentum is zero, $\bar{u} = 0$.  Since the equation is $0$-homogeneous in $f$, we can assume that the total mass of the flock is given by $M=|\O^n|$. So, the corresponding Maxwellian is given by 
\begin{equation}\label{e:Max0}
\mu = \frac{1}{(2\pi \s)^{n/2}} e^{- \frac{|v|^2}{2\s}}.
\end{equation}

Let us now introduce several key quantities.  The central quantity is the relative entropy
\begin{equation}\label{e:entropy}
\cH(f|\mu) = \s \int_{\domain} f \log \frac{f}{\mu} \dv \dx,
\end{equation}
or more explicitly,
\[
\cH(f|\mu) = \s \int_{\domain} f \log f \dv \dx + \frac{1}{2} \int_{\domain} |v|^2 f \dv\dx + \s |\O^n| \frac{n}{2}\log(2\pi \s).
\]
According to the \CK\ we have
\begin{equation}\label{ }
c \| f - \mu \|_1^2\leq \cH(f|\mu).
\end{equation}
So, to prove \thm{t:main} it suffices to establish an exponential bound on the entropy itself.

We will work with a hierarchy of energies\footnote{We intentionally leave out the $\frac12$ factor in order to simply formulas that follow.}:
\begin{equation*}\label{}
\begin{split}
E & =  \int_{\domain} |v|^2 f \dv\dx, \\
\cE & = \int_{\O^n} \rho |u|^2  \dx, \\
\cE_\phi &= \int_{\O^n} \frac{(u \rho)^2_\phi}{\rho_\phi} \dx, \\
 \cE_{\phi \phi} & = \int_{\O^n} \rho |\ufilt|^2 \dx.
\end{split}
\end{equation*} 
\begin{claim} We have
\begin{equation}\label{e:enhier}
 \cE_{\phi \phi} \leq  \cE_{\phi} \leq \cE \leq E.
\end{equation}
\end{claim}
The last inequality is the classical maximization principle. The rest follow by application of the H\"older and Minkowski inequalities.

The difference between the two mollified macroscopic energies will play a special role in the analysis, 
\[
\cA = \cE_\phi -\cE_{\phi \phi}.
\]
It represents a quantitative measure for alignment for the Favre filtered field $\uFavre$ as will be elaborated in \lem{l:Arep}.
 
Next, for a macroscopic field $u$ we consider the partial Fisher information centered at $u$
\begin{equation}\label{}
\cI_{vv}(f,u) = \int_\domain \frac{\left| \s \n_v f+ (v - u)f \right|^2}{f} \dv \dx.
\end{equation}
Pertaining to the situation when $u = \ufilt$ we observe the identity which follows by a simple expansion of the numerator:
\begin{equation}\label{e:Fishident}
\cI_{vv}(f,\ufilt) = \cI_{vv}(f,0) + \cE_{\phi \phi} - 2\cE_\phi.
\end{equation}

\begin{lemma}
We have the following two forms of the entropy law:
\begin{align}
\ddt \cH(f|\mu) & = - \cI_{vv}(f,\ufilt) - \cA, \label{e:elaw} \\
\ddt \cH(f|\mu) & = -\cI_{vv}(f,0) +  \cE_\phi. \label{e:elaw0} 
\end{align}
\end{lemma}
The proof of \eqref{e:elaw} goes by a direct verification. Then \eqref{e:elaw0} follows from \eqref{e:elaw}  and  \eqref{e:Fishident}. 

From \eqref{e:elaw} we can see that the FPA equation has a globally decaying entropy. One not so obvious consequence of this  is that the full energy $E$ by itself remains uniformly bounded.  To see that one has to circumvent the issue of fact that the Boltzmann functional $\int f \log f \dx \dv$  is not sign-definite. This was addressed in \cite{GJV2004} by showing that there is an absolute constant $C>0$ such that 
\[
\int_\domain | f \log f | \dx \dv \leq \int_\domain f \log f \dx \dv + \frac14 \int_\domain |v|^2 f  \dx \dv + C \leq C'_\s \cH + C''_\s.
\]
So, the energy alone is also bounded,
\begin{equation}\label{e:Eunif}
E(t) \leq C'_\s \cH_0 + C''_\s, \qquad \forall t>0.
\end{equation}

From version \eqref{e:elaw0} of the law, which links the information directly to $\mu$, see next section, we can identify two major obstacles in establishing coercivity directly -- the traditional lack of dissipation in the $x$-variable, and an additional macroscopic energy that comes from the alignment force. The two will be handled simultaneously in the next section.

\section{Hypocoercivity}\label{s:hypo}

We now get to the main proof of Theorems \ref{t:main} and \ref{t:sid}.

It is more technically convenient to recast the FPA equation in terms of the renormalized distribution $h = \frac{f}{\mu}$, which satisfies
\begin{equation}\label{e:FPAh}
h_t + v \cdot \n_x h = \s \D_v h - v \cdot \n_v h  - \ufilt \cdot \n_v h + \s^{-1}(\ufilt \cdot v) h.
\end{equation}
The Fokker-Planck part of the equation \eqref{e:FPAh} has the traditional structure of an evolution semigroup given by the generator
\[
L =  \s A^* A + B, \qquad B = v \cdot \n_x, \quad A =  \n_v, \quad A^* = ( \s^{-1} v -  \n_v) \cdot.
\]
Here the adjoint is understood with respect to the inner product of the weighted space $L^2(\mu)$:
\[
\lan g_1 g_2 \ran = \int_\domain g_1 g_2 \dmu, \quad \dmu = \mu \dv \dx.
\]

The nonlinear alignment part can be represented in terms of the action of $A^*$ :
\[
A^*(\ufilt h) =   - \ufilt \cdot \n_v h + \s^{-1} (\ufilt \cdot v) h.
\]
Thus, \eqref{e:FPAh} can be written concisely as
\begin{equation}\label{e:FPALN}
h_t = - Lh +A^*(\ufilt h).
\end{equation}

We now rewrite all the entropic quantities in terms of $h$:
\[
\cH(h) = \s \int_\domain h \log h \dmu, \quad \cI_{vv}(h) = \cI_{vv}(f,0) = \s^2 \int_\domain \frac{|\n_v h|^2}{h} \dmu,
\]
and consider two additional information functionals
\[
 \cI_{xv}(h) = \s^{3/2} \int_\domain \frac{\n_x h \cdot \n_v h}{h} \dmu, \quad  \cI_{xx}(h) = \s \int_\domain \frac{|\n_x h|^2}{h} \dmu.
 \]
 Recall that the sum $\cI =  \cI_{vv}  + \cI_{xx}$ constitutes the full Fisher information 
and, by the classical (rescaled) log-Sobolev inequality, see \cite{Gross1975}, controls the relative entropy:
\begin{equation}\label{e:logSob}
 \cI(h) \geq \l \cH(h), 
\end{equation}
where $\l>0$ is independent of $\s$.
 
In the following three lemmas we will calculate evolution laws for each of the information functionals.   As will be seen, the alignment component has some remarkable cancellations and interacts closely with the Fokker-Planck part.  We also protocol dependence on $\s$ which is essential in proving \thm{t:sid} later.

\begin{lemma}\label{}
We have
\[
\ddt \cI_{vv}(h)  =  -2 \s^3 \cD_{vv}-  2  \cI_{vv} -2 \s^{1/2} \cI_{xv} + 2 \s \cE_\phi,
\]
where
\[
\cD_{vv} = \lan h |\n_v^2 \bar{h}|^2 \ran, \qquad \bar{h} = \log h.
\]
\end{lemma}
\begin{proof}
Let us write $\cI_{vv} = \lan \n_v h, \n_v \bar{h} \ran$. Computing the derivative  we obtain
\begin{equation*}\label{}
\begin{split}
 \frac{1}{\s^2} \ddt \cI_{vv} = 2\lan \n_v h_t \cdot \n_v \bar{h} \ran - \lan |\n_v \bar{h} |^2 h_t \ran  & = -2 \s \lan \n_v A^*A h \cdot  \n_v \bar{h} \ran + \s \lan |\n_v \bar{h} |^2 A^*A h \ran \\
&-2 \lan \n_v B h\cdot  \n_v \bar{h} \ran + \lan |\n_v \bar{h} |^2 B h \ran \\
& + 2\lan \n_v A^*(\ufilt h)\cdot \n_v \bar{h} \ran - \lan |\n_v \bar{h} |^2 A^*(\ufilt h) \ran\\
& = J_A + J_B + J_u.
\end{split}
\end{equation*}
Let us start with the $A$-part. Observe that 
\[
\p_{v_i} (A^*A h) = A^*Ah_{v_i} + \s^{-1}h_{v_i}.
\]
Thus,  adopting Einstein's summation convention:
\begin{equation*}\label{}
\begin{split}
J_A & = -2 \s \lan A^*Ah_{v_i} \bar{h}_{v_i} \ran - 2 \cI_{vv} + \s \lan |\n_v \bar{h} |^2 A^*A h \ran\\
&= -2 \s \lan Ah_{v_i} \cdot A \bar{h}_{v_i} \ran - 2  \cI_{vv} + \s \lan A |\n_v \bar{h} |^2\cdot  A h \ran\\
&= -2\s \lan h A\bar{h}_{v_i}\cdot  A \bar{h}_{v_i} \ran -2\s \lan \bar{h}_{v_i} Ah\cdot  A \bar{h}_{v_i} \ran - 2 \cI_{vv} + 2 \s \lan \bar{h}_{v_i}A \bar{h}_{v_i} \cdot  A h \ran.
\end{split}
\end{equation*}
The second and  last terms cancel, while the first one involves the sum of the squares of all second order derivatives $|\n_v^2 \bar{h}|^2$. We arrive at
\[
J_A = -2 \s \cD_{vv}  -  2 \cI_{vv}.
\]
Next,
\[
J_B = -2 \lan  \n_x h \cdot  \n_v \bar{h} \ran - 2 \lan (v \cdot \n_x h_{v_i}) \bar{h}_{v_i} \ran  + \lan |\n_v \bar{h} |^2 v \cdot \n_x h \ran.
\] 
Let us look into the middle term:
\[
- 2 \lan v \cdot \n_x h_{v_i} \bar{h}_{v_i} \ran = - 2 \lan v \cdot \n_x h_{v_i} {h}_{v_i}{h}^{-1} \ran  = - \lan v \cdot \n_x |h_{v_i}|^2 h^{-1} \ran  = -\lan |h_{v_i}|^2 v \cdot \n_x h h^{-2} \ran =-\lan |\bar{h}_{v_i}|^2 v \cdot \n_x h \ran
\]
which cancels the last term in the previous formula. So,
\[
J_B = -2  \cI_{xv}.
\]
The proof of the lemma is concluded by the following exact identity:
\[
J_u = 2 \s^{-1} \cE_\phi.
\]
To prove it we manipulate with the formula for $J_u$ as follows
\begin{equation*}\label{}
\begin{split}
J_{u} & = 2\lan \n_v A^*(\ufilt h)\cdot  \n_v \bar{h} \ran - \lan |\n_v \bar{h} |^2 A^*(\ufilt h) \ran \\
& = 2\lan \n_v ( \s^{-1} v \cdot \ufilt h - \ufilt \cdot \n_v h)\cdot  \n_v \bar{h} \ran - \lan \n_v |\n_v \bar{h} |^2 \cdot  \ufilt h \ran\\
& = 2\lan [ \s^{-1}  \ufilt h + \s^{-1}  (v \cdot \ufilt) \n_v h - \n^2_v h (\ufilt) ] \cdot  \n_v \bar{h} \ran - 2\lan \n_v^2 \bar{h}(\n_v \bar{h}) \cdot \ufilt h \ran.
\end{split}
\end{equation*}
 where $\n_v^2 h$ is the Hessian matrix of $h$. Notice that first part of the first term produces the mollified energy,
 \[
 \lan \ufilt h \cdot  \n_v \bar{h} \ran =  \int_{\domain} \ufilt \cdot v h \dmu =\int_{\domain} \ufilt \cdot v f \dx \dv = \int_{\O^n} \ufilt \cdot (u\rho) \dx = \cE_\phi.
 \] 
 We now show that the remaining part of $J_{u} $ vanishes.  Indeed, using that 
 \begin{equation}\label{e:hbarh}
 \n_v^2 h = h \n_v^2 \bar{h} + \frac{1}{h} \n_v h \otimes \n_v h
\end{equation}
 we obtain 
 \begin{equation*}\label{}
 2\lan \s^{-1}  (v \cdot \ufilt) \n_v h - h \n^2_v \bar{h} (\ufilt) - \frac{1}{h}(\ufilt \cdot \n_v h) \n_v h, \n_v \bar{h} \ran - 2\lan \n_v^2 \bar{h}(\n_v \bar{h}), \ufilt h \ran 
\end{equation*}
 {by the symmetry of the Hessian,}
 \begin{equation}\label{e:auxA}
 =  2\lan \s^{-1}  (v \cdot \ufilt) \n_v h - \frac{1}{h}(\ufilt \cdot \n_v h) \n_v h, \n_v \bar{h} \ran - 4\lan \n_v^2 \bar{h}(\n_v \bar{h}), \ufilt h \ran.
\end{equation}
Looking at the first bracketed term, we can interpret it as an action of $A^*$:
\begin{equation*}\label{}
\begin{split}
&2\lan \s^{-1}  (v \cdot \ufilt) \n_v h - \frac{1}{h}(\ufilt \cdot \n_v h) \n_v h, \n_v \bar{h} \ran = 2\lan \s^{-1}  h (v \cdot \ufilt)- h(\ufilt \cdot \n_v h), | \n_v \bar{h}|^2 \ran \\
= &2 \lan A^*(\ufilt h), |\n_v \bar{h} |^2 \ran = 2 \lan \ufilt h , \n_v  |\n_v \bar{h} |^2\ran = 4\lan \n_v^2 \bar{h}(\n_v \bar{h}), \ufilt h \ran.
\end{split}
\end{equation*}
 which cancels the second bracketed term in \eqref{e:auxA}.

\end{proof}
\begin{lemma}\label{}
We have
\[
\ddt \cI_{xv}(h)  \leq - \frac12 \s^{1/2} \cI_{xx}  - \s \cI_{xv} +C \s^{3/2} \sqrt{ \cD_{vv} \cE_\phi}+ \s^{5/2} \sqrt{ \cD_{xv} \cD_{vv}} + \frac{1}{2}\s^{3/2} \cE_\phi,
\]
where
\[
\cD_{xv} = \lan h |\n_v \n_x \bar{h}|^2 \ran.
\]
\end{lemma}
\begin{proof}
Let us express the derivative as follows
\[
\frac{1}{\s^{3/2}} \ddt \cI_{xv}(h) = \lan \n_x h_t \cdot \n_v \bar{h} \ran + \lan \n_x \bar{h} \cdot \n_v {h}_t \ran - \lan h_t \n_v \bar{h}\cdot \n_x \bar{h} \ran: = J_A + J_B + J_u,
\]
where as before $J_A,J_B,J_u$ collect contributions from $A^*A$, $B$, and alignment components, respectively. We start with the easier terms:
\begin{equation}\label{e:JBxv}
J_B = - \lan \n_x (v \cdot \n_x h) \cdot \n_v \bar{h} \ran - \lan \n_x \bar{h} \cdot \n_v (v \cdot \n_x h) \ran + \lan (v \cdot \n_x h) \n_v \bar{h}\cdot \n_x \bar{h} \ran.
\end{equation}
The middle term can be  expanded as follows
\[
 - \lan \n_x \bar{h} \cdot \n_v (v \cdot \n_x h) \ran = -  \cI_{xx} - \lan \bar{h}_{x_i} v_j h_{x_j v_i}\ran,
 \]
 integrating by parts in $x_j$,
 \[
 = -  \cI_{xx} +\lan \bar{h}_{x_i x_j} v_j h_{v_i}\ran
\]
using that $\bar{h}_{x_i x_j}= h^{-1} {h}_{x_i x_j} - h^{-2} h_{x_i} h_{x_j}$,
\[
= -  \cI_{xx} +\lan{h}_{x_i x_j} v_j \bar{h}_{v_i}\ran - \lan \bar{h}_{x_i} \bar{h}_{x_j} v_j h_{v_i} \ran
\]
and the last two terms cancel with the first and third terms in \eqref{e:JBxv}. Thus,
\[
J_B = -  \cI_{xx} .
\]

Next, we examine the $J_u$-term:
\begin{equation*}\label{}
\begin{split}
J_u & =  \lan \n_x A^*(\ufilt h) \cdot \n_v \bar{h} \ran + \lan \n_x \bar{h} \cdot \n_v  A^*(\ufilt h) \ran - \lan  A^*(\ufilt h) \n_v \bar{h}\cdot \n_x \bar{h} \ran \\
& = \lan  A^*((\ufilt)_{x_i} h)   \bar{h}_{v_i} \ran + \lan  A^*(\ufilt h_{x_i})  \bar{h}_{v_i} \ran \\
& + \lan  \bar{h}_{x_i}    A^*(\ufilt h_{v_i}) \ran + \lan \n_x \bar{h} \cdot   \ufilt h \ran \\
&- \lan  h \ufilt \cdot \n_v (\n_v \bar{h}\cdot \n_x \bar{h}) \ran \\
&= \lan  (\ufilt)_{x_i} h \cdot \n_v  \bar{h}_{v_i} \ran+ \lan  \ufilt h_{x_i} \cdot \n_v \bar{h}_{v_i} \ran + \lan  \n_v \bar{h}_{x_i}   \cdot \ufilt h_{v_i} \ran \\
& +  \lan \n_x \bar{h} \cdot   \ufilt h\ran- \lan  h \ufilt \cdot \n_v (\n_v \bar{h}\cdot \n_x \bar{h}) \ran .
\end{split} 
\end{equation*}
Note that 
\[
\lan  \ufilt h_{x_i} \cdot \n_v \bar{h}_{v_i} \ran + \lan  \n_v \bar{h}_{x_i}   \cdot \ufilt h_{v_i} \ran = \lan  h \ufilt \cdot \n_v (\n_v \bar{h}\cdot \n_x \bar{h}) \ran,
\]
so, those two terms will cancel with the last one. Thus,
\[
J_u = \lan  (\ufilt)_{x_i} h \cdot \n_v  \bar{h}_{v_i} \ran+ \lan \n_x \bar{h} \cdot   \ufilt h\ran
\]
We estimate the first term as follows:
\[
\lan  (\ufilt)_{x_i} h \cdot \n_v  \bar{h}_{v_i} \ran \leq   \cD^{1/2}_{vv} \lan | (\ufilt)_{x_i}|^2 h \ran^{1/2}.
\]
Denoting $\psi_i = |\p_{x_i} \phi|$, and in view of \eqref{e:rlowfromdens}, we obtain
\begin{equation}\label{e:uest1}
\lan | (\ufilt)_{x_i}|^2 h \ran = \int_{\O^n}  | (\ufilt)_{x_i}|^2 \rho \dx \leq  \int_{\O^n}  (|\uFavre|^2)_{\psi_i} \rho \dx
=  \int_{\O^n}  |\uFavre|^2\rho_{\psi_i} \dx  \leq  \frac{|\O^n|}{c} \|\psi_i\|_\infty  \cE_\phi.
\end{equation}
Thus,
\[
\lan  (\ufilt)_{x_i} h \cdot \n_v  \bar{h}_{v_i} \ran \leq   C \sqrt{\cD_{vv} \cE_\phi}.
\]

Turning to the remaining term, we obtain
\begin{equation*}\label{}
\begin{split}
\lan \n_x \bar{h} \cdot   \ufilt h\ran  \leq \frac{1}{2} \lan |\n_x \bar{h}|^2  h\ran + \frac{1}{2} \lan |\ufilt|^2 h \ran \leq \frac{1}{2}  \cI_{xx} + \frac{1}{2}\lan |\uFavre|^2 h_\phi \ran = \frac{1}{2} \cI_{xx} + \frac{1}{2} \int_{\O^n} |\uFavre|^2 \rho_\phi \dx  = \frac{1}{2} \cI_{xx} + \frac{1}{2}\cE_\phi.
\end{split}
\end{equation*}
In summary,
\begin{equation}\label{}
J_u + J_B \leq -   \frac{1}{2}    \cI_{xx} +  C \sqrt{\cD_{vv} \cE_\phi} +\frac{1}{2}\cE_\phi.
\end{equation}

Finally let us look into the $J_A$-term:
\[
\frac{1}{\s} J_A =   - \lan \n_x A^*Ah \cdot \n_v \bar{h} \ran - \lan \n_x \bar{h} \cdot \n_v  A^*Ah \ran + \lan  A^*Ah \n_v \bar{h}\cdot \n_x \bar{h} \ran = I + II + III.
\]
For $I$ we obtain
\[
I = - \lan A^* A h_{x_i} \bar{h}_{v_i} \ran = - \lan \n_v h_{x_i} \cdot \n_v  \bar{h}_{v_i} \ran= -  \lan h \n_v \bar{h}_{x_i} \cdot \n_v  \bar{h}_{v_i} \ran -  \lan h_{x_i} \n_v h \cdot \n_v  \bar{h}_{v_i} \ran.
\]
For $II$ we obtain
\[
II = - \lan \n_x \bar{h} \cdot \n_v h \ran- \lan  \bar{h}_{x_i} A^*A {h}_{v_i} \ran  =  -  \cI_{xv} -  \lan h \n_v \bar{h}_{x_i} \cdot \n_v  \bar{h}_{v_i} \ran -  \lan  \n_v \bar{h}_{x_i} \cdot \n_v  h \bar{h}_{v_i} \ran .
\] 
The two add up to 
\[
\begin{split}
I+II & =  -  \cI_{xv} -  \lan h \n_v \bar{h}_{x_i} \cdot \n_v  \bar{h}_{v_i} \ran - \lan  Ah \cdot A(\n_v \bar{h}\cdot \n_x \bar{h} )\ran \\
& \leq  -  \cI_{xv} + \sqrt{ \cD_{xv} \cD_{vv}} - III.
\end{split}
\]
Thus,
\[
J_A \leq - \s \cI_{xv} + \s \sqrt{ \cD_{xv} \cD_{vv}}.
\]
\end{proof}

\begin{lemma}\label{}
We have
\[
\ddt \cI_{xx}(h)  \leq   - \s^2 \cD_{xv}+C \cE_\phi.
\]
\end{lemma}
 \begin{proof}
 We have
 \[
\frac{1}{\s} \ddt \cI_{xx}(h) =  2\lan \n_x h_t \cdot \n_x \bar{h} \ran - \lan |\n_x \bar{h} |^2 h_t \ran .
\]

The contribution from the $B$-term cancels entirely:
\[
\begin{split}
J_B & = -2 \lan \n_x(v \cdot \n_x h) \cdot \n_x \bar{h} \ran + \lan |\n_x \bar{h} |^2 v \cdot \n_x h \ran \\
& = -2 \lan (v \cdot \n_x h_{x_i}) h_{x_i} h^{-1} \ran + \lan |\n_x \bar{h} |^2 v \cdot \n_x h \ran \\
& = - \lan (v \cdot \n_x |\n_x h|^2 h^{-1} \ran + \lan |\n_x \bar{h} |^2 v \cdot \n_x h \ran\\
& =  - \lan v \cdot \n_x h  |\n_x h|^2 h^{-2} \ran + \lan |\n_x \bar{h} |^2 v \cdot \n_x h \ran = 0.
\end{split}
\]

Turning, next, to the $A$-term we obtain
\[
\begin{split}
\frac{1}{\s} J_A & = -2 \lan \n_x A^*A h \cdot \n_x \bar{h} \ran + \lan |\n_x \bar{h} |^2 A^*Ah \ran \\
& = -2 \lan A h_{x_i} \cdot A \bar{h}_{x_i} \ran + \lan A |\n_x \bar{h} |^2 \cdot Ah \ran \\
& = -2 \lan \n_v (h \bar{h}_{x_i}) \cdot \n_v \bar{h}_{x_i} \ran + \lan \n_v |\n_x \bar{h} |^2 \cdot \n_vh \ran\\
& = -2 \lan h \n_v \bar{h}_{x_i} \cdot \n_v \bar{h}_{x_i} \ran -2  \lan \bar{h}_{x_i} \n_v h \cdot \n_v \bar{h}_{x_i} \ran + \lan \n_v |\n_x \bar{h} |^2 \cdot \n_vh \ran\\
& = -2 \cD_{xv} -  \lan \n_v |\n_x \bar{h} |^2 \cdot \n_vh \ran+ \lan \n_v |\n_x \bar{h} |^2 \cdot \n_vh \ran = -2 \cD_{xv}.
\end{split}
\]
Thus,
\[
J_A = -2 \s \cD_{xv}.
\]

Finally, the alignment term is given by
\[
J_u = 2\lan \n_x A^*(\ufilt h) \cdot \n_x \bar{h} \ran - \lan |\n_x \bar{h} |^2 A^*(\ufilt h) \ran .
\]
In the second term we simply swap the operator $A^*$: 
\[
- \lan |\n_x \bar{h} |^2 A^*(\ufilt h) \ran = - \lan \n_v |\n_x \bar{h} |^2 \ufilt h \ran.
\]
The first term is
\begin{equation*}\label{}
\begin{split}
2\lan \n_x A^*(\ufilt h) \cdot \n_x \bar{h} \ran & =2\lan  A^*(\ufilt h_{x_i})  \bar{h}_{x_i} \ran +2\lan  A^*((\ufilt)_{x_i} h) \bar{h}_{x_i} \ran \\
&= 2\lan  h  \bar{h}_{x_i} \ufilt\cdot \n_v \bar{h}_{x_i} \ran+2\lan  (\ufilt)_{x_i} h  \cdot\n_v \bar{h}_{x_i} \ran\\
& =  \lan \n_v |\n_x \bar{h} |^2 \ufilt h \ran +2\lan  (\ufilt)_{x_i} h  \cdot\n_v \bar{h}_{x_i} \ran.
\end{split}
\end{equation*}
We can see that the first term cancels with the previous one. As to the last one we estimate
\[
2\lan  (\ufilt)_{x_i} h  \cdot\n_v \bar{h}_{x_i} \ran \leq  2 \cD_{xv}^{1/2} \lan  |(\ufilt)_{x_i}|^2 h \ran^{1/2},
\]
while the  term $\lan  |(\ufilt)_{x_i}|^2 h \ran$ has been estimated previously in \eqref{e:uest1}. Thus,
\begin{equation}\label{}
J_u \leq C \sqrt{\cD_{xv}\cE_\phi} \leq \s \cD_{xv} + C_1 \s^{-1} \cE_\phi.
\end{equation}
Summing up the obtain estimates proves the result.

\end{proof}
 
 Denoting
\[
\tilde{\cI} = \cI_{vv} + \cI_{xv} + \cI_{xx}
\]
and noticing that 
\[
| \cI_{xv}| \leq \frac12( \cI_{vv}  + \cI_{xx}),
\]
we can see that $\tilde{\cI}$ is comparable to the full Fisher information $
 \frac12\cI \leq \tilde{\cI} \leq \frac32\cI$. As such, by the log-Sobolev inequality \eqref{e:logSob}, we have
\begin{equation}\label{e:logSobI}
\tilde{\cI} \geq  \frac{\l}{2} \cH(h).
\end{equation}

Let us now add the estimates from all three lemmas:
\[
\ddt \tilde{\cI}  \leq  -2 \s^3 \cD_{vv}-  2  \cI_{vv} -2 \s^{1/2} \cI_{xv} - \frac12 \s^{1/2} \cI_{xx}  - \s \cI_{xv} +C \s^{3/2} \sqrt{ \cD_{vv} \cE_\phi}+ \s^{5/2} \sqrt{ \cD_{xv} \cD_{vv}}   - \s^2 \cD_{xv}+C \cE_\phi.
\]
Using that  
\[
\s^{5/2} \sqrt{ \cD_{xv} \cD_{vv}} \leq \frac12( \s^2 \cD_{xv} + \s^3 \cD_{vv})
\]
 and that 
 \[
 C \s^{3/2} \sqrt{ \cD_{vv} \cE_\phi}  \leq \s^3 \cD_{vv} + {C} \cE_\phi,
 \]
  we can see that all the dissipative terms are in negative, and we further estimate (with possibly different $C$)
\[
\ddt \tilde{\cI}  \leq -  2 \cI_{vv} - (2 \s^{1/2} + \s) \cI_{xv} - \frac12 \s^{1/2} \cI_{xx}  + C \cE_\phi.
\]
By generalized Young's inequality, the mixed information is estimated by 
\[
(2 \s^{1/2} + \s) \cI_{xv} \leq \frac14  \s^{1/2} \cI_{xx} + C(\s_0) \cI_{vv}.
\]
Thus,
\begin{equation}\label{e:Iprelim}
\ddt \tilde{\cI}  \leq C_1(\s_0) \cI_{vv} - \frac14 \s^{1/2} \cI_{xx}+C_2(\s_0,\O^n,\phi,\d) \cE_\phi.
\end{equation}
 
To absorb $\cI_{vv}$ and the energy $ \cE_\phi$ we invoke the entropy laws \eqref{e:elaw} - \eqref{e:elaw0}. But before we do that let us take a closer look at the alignment term $\cA$.

\begin{lemma}\label{l:Arep}
We have the following formula:
\begin{equation*}\label{}
\begin{split}
\cA & = \frac12 \int_{\O^n \times \O^n} \rho_{\phi \phi}(x,y) | \uFavre(x) - \uFavre(y)|^2 \dx \dy, \\
\rho_{\phi \phi}(x,y) & = \int_{\O^n} \phi(\xi - x)\phi(\xi - y) \rho(\xi) \dxi.
\end{split}
\end{equation*}
\end{lemma}
\begin{proof}
The proof consists of the following streak of identities:
\begin{equation*}\label{}
\begin{split}
\cA & = \int_{\O^n} (\rho_\phi |\uFavre|^2 - \rho |(\uFavre)_\phi |^2 ) \dx\\
& = \int_{\O^n} (\rho_\phi \uFavre \cdot \uFavre - \rho (\uFavre)_\phi \cdot (\uFavre)_\phi  ) \dx\\
& = \int_{\O^n} (\rho_\phi \uFavre - (\rho (\uFavre)_\phi)_\phi  )  \cdot \uFavre \dx\\
& = \int_{\O^n \times \O^n} \phi(x - \xi) \rho(\xi)( \uFavre(x) - (\uFavre)_\phi(\xi)  )  \cdot \uFavre(x) \dxi \dx\\
& = \int_{\O^n \times \O^n\times \O^n} \phi(x - \xi)\phi(y - \xi)  \rho(\xi)( \uFavre(x) - \uFavre(y)  )  \cdot \uFavre(x) \dxi \dx \dy\\
& = \int_{\O^n \times \O^n} \rho_{\phi \phi}(x,y)( \uFavre(x) - \uFavre(y)  )  \cdot \uFavre(x) \dx \dy\\
& = \frac12 \int_{\O^n \times \O^n} \rho_{\phi \phi}(x,y) | \uFavre(x) - \uFavre(y)|^2 \dx \dy,
\end{split}
\end{equation*}
where in the last step we performed symmetrization in $x,y$.
\end{proof}

Next, we show that the alignment term controls the mollified energy itself.
\begin{lemma}\label{}
Under the assumption \eqref{e:hc} there exists a constant $c > 0$ depending on $\d$ and parameters of the system such that 
\[
\cA \geq c  \cE_\phi.
\]
\end{lemma}
\begin{proof}
First, note that if $|x-y|<r_0 - r$, then
\[
\rho_{\phi \phi}(x,y) = \int_{\O^n} \phi(\eta) \phi(y-x+\eta) \rho(x-\eta) \deta \geq c_0^2 \int_{|\eta|<r} \rho(\xi) \dxi \geq c_0^2 \d.
\]
Consequently,
\[
 \rho_{\phi \phi}(x,y)  \geq c \one_{|x-y| < r_0 - r}.
\]
Thus,
\[
\cA \geq c  \int_{|x-y| < r_0 -r}  | \uFavre(x) - \uFavre(y)|^2 \dx \dy.
\]
We now invoke \cite[Lemma 2.1]{LSlimiting}, to claim that 
\[
\int_{|x-y| < r_0 -r}  | \uFavre(x) - \uFavre(y)|^2 \dx \dy \geq c(r_0,r) \| \uFavre - \overline{\uFavre} \|_{L^2(\O^n)}^2,
\]
where $\overline{\uFavre}$ is the mean value of the Favre-filtered velocity.  However, using the assumed zero momentum, $\overline{(u\rho)_\phi} = 0$, we estimate
\begin{equation*}\label{}
\begin{split}
|\O^n| \|\phi \|_\infty \| \uFavre - \overline{\uFavre} \|_{L^2(\O^n)}^2 & \geq \int_{\O^n} \rho_\phi |\uFavre -   \overline{\uFavre}|^2 \dx =   \int_{\O^n} \rho_\phi |\uFavre |^2 \dx - \underbrace{ 2 \overline{\uFavre} \cdot \overline{(u\rho)_\phi}}_{=0} + |\O^n| |\overline{\uFavre}|^2 \\
&=  \cE_\phi + |\O^n| |\overline{\uFavre}|^2 \geq \cE_\phi ,
\end{split}
\end{equation*} 
and the lemma follows.
\end{proof}

\begin{proof}[Proof of \thm{t:main}]
Let us go back to the entropy law \eqref{e:elaw} where we drop the information and use the control bound on the alignment:
\[
\ddt \cH \leq -c \cE_\phi.
\]
Combining with \eqref{e:elaw0} we obtain
\[
\left( \frac{1}{c} + 1\right) \ddt \cH \leq - \cI_{vv}.
\]
Thus, together with \eqref{e:Iprelim},
\[
\ddt \left[ C_3 \cH+ \tilde{ \cI} \right] \leq - C_4 \s^{1/2}\tilde{ \cI}.
\]
Recalling the log-Sobolev inequality \eqref{e:logSobI}, we further conclude
\begin{equation}\label{e:mainGrown}
\ddt \left[ C_3 \cH+ \tilde{ \cI} \right] \leq - C_5 \s^{1/2}  \left[ C_3 \cH+ \tilde{ \cI} \right],
\end{equation}
and \thm{t:main} follows.
\end{proof}

\begin{proof}[Proof of \thm{t:sid}]
Recall that we work under the assumption that $M = |\O^n|$, so the uniform distribution $1$ has the same mass as our density $\rho$. Let us also observe that  if a density $\rho$ is sufficiently close to be uniform in $L^1$-metric
\begin{equation}\label{e:L1rho}
\|\rho - 1\|_{L^1}^2 < \e_0,
\end{equation}
for some small $\e_0>0$, then $\rho$ is hydrodynamically connected at scale $r_0/2$ in the sense of Definition~\ref{d:hc} with $\d = \frac12$.  And in particular $\rho_\phi \geq c$, where $c$ depend only on the parameters of the model.

 Let us assume that 
$\cI(f_0) \leq \s \e$, where $\e$ is to be determined later.  By the log-Sobolev inequality \eqref{e:logSobI0}, the maximization principle, and \CK\ we obtain
\[
\e \s \geq \l \s \int_{\O^n} \rho_0 \log \rho_0 \dx \geq c \s \|\rho_0 - 1\|_{L^1}^2.
\]
Thus, $\|\rho_0 - 1\|_{L^1}^2 \leq c \e$, and if $c\e < \e_0/2$, the above discussion implies that a solution will exist on a time interval $[0,T)$ according to \thm{t:lwp}. By continuity, we can assume that on the same interval inequality 
\eqref{e:L1rho} still holds. In particular this fulfills the continuation criterion of \thm{t:lwp} and the solution can be continued until the condition \eqref{e:L1rho} is violated. Let $T^*$ be the first such time. This implies bound \eqref{e:mainGrown} on the same interval $[0,T^*)$ with all $C$'s dependent only on the parameters of the model. Hence,
\begin{equation}\label{e:HCI}
C_3 \cH(t) + \tilde{\cI} (t) \leq  C_3 \cH_0 + \tilde{\cI}_0 \leq ( C_3 \l^{-1} + 1) \tilde{\cI}_0 \leq C_4 \e \s.
\end{equation}
By the same streak of inequalities as applied initially, we obtain
\[
\|\rho - 1\|_{L^1}^2 < C_5 \e < \e_0/2,
\]
if $\e$ is chosen small enough. This implies that $T^* = \infty$ and proves the result.\end{proof}

\section{Hydrodynamic limit}\label{s:hydro}

Let us note in passing that accentuating the mollified alignment term as in  
\begin{equation}\label{e:FPAerror}
\p_t f^\e + v\cdot \n_x f^\e = \frac{1}{\e} [ \D_v f^\e + \n_v((v - \ufilt^\e)f^\e ) ], 
\end{equation}
would result in a trivial limit as $\e \to 0$. Indeed, from the law \eqref{e:elaw} we infer that $\cA \to 0$ for any $t>0$. So, at least for non-vacuous solutions this  implies $u \to \bar{u}$, and hence in the limit we obtain a perfectly aligned solution. 

By taking the moments of \eqref{e:FPA}, however, we can see that the natural macroscopic system obtained by the Maxwellian closure reads
\begin{equation}\label{e:macro}
\begin{split}
\rho_t  + \n \cdot (u \rho) & = 0\\
(\rho u)_t + \n \cdot (\rho u \otimes u) +  \n \rho &=  \rho( \ufilt - u).
\end{split}
\end{equation}
In this section we will justify \eqref{e:macro} via a similar approach taken in \cite{KMT2015,MV2008} by considering an equation with penalization forcing
\begin{equation}\label{e:FPAe}
\p_t f^\e + v\cdot \n_x f^\e = \frac{1}{\e} [ \D_v f^\e + \n_v((v - u^\e)f^\e ) ] + \n_v((v - \ufilt^\e)f^\e )  ,
\end{equation}
where $u^\e = (u\rho)^\e / \rho^\e$ is the usual macroscopic velocity field associated with $f^\e$.  The issues of well-posedness for \eqref{e:FPAe} are very similar to the ones encountered in the Cucker-Smale case, see \cite{KMT2013}, and will not be addressed here. The local alignment term already contains all the major issues associated with roughness of the field, which are even less severe for the filtered field $\ufilt$. Thus, we will work in the settings of weak solutions to \eqref{e:FPAe} which satisfy the corresponding entropy laws elaborated in \lem{l:elawepsilon}.

The macroscopic system is given by
\begin{equation}\label{e:macroe}
\begin{split}
\rho^\e_t  + \n \cdot (u^\e \rho^\e) & = 0\\
(\rho^\e u^\e)_t + \n \cdot (\rho^\e u^\e \otimes u^\e) +  \n \rho^\e + \n_x \cdot \cR_\e &=  \rho^\e( \ufilt^\e - u^\e)\\
\cR_\e & = \int_{\R^n} ( v \otimes v - u^\e \otimes u^\e - \I ) f^\e \dv.
\end{split}
\end{equation}

Let us set our notation first. We consider the local Maxwellian associated with the limiting solution
\begin{equation}\label{e:MaxLim}
\mu =\frac{ \rho(x,t) }{(2\pi)^{n/2}} e^{- \frac{|v - u(x,t)|^2}{2}},
\end{equation}
and Maxwellian of the solution to \eqref{e:FPAe},
\begin{equation}\label{ }
\mu^\e =\frac{ \rho^\e(x,t) }{(2\pi)^{n/2}} e^{- \frac{|v - u^\e(x,t)|^2}{2}}.
\end{equation}
We now consider the kinetic entropy of $f^\e$ relative to the limiting Maxwellian 
\[
\cH(f^\e | \mu) = \int_{\domain} f^\e\log \frac{f^\e}{\mu} \dv \dx.
\]
A simple identity shows that it controls the entropy of $f^\e$ relative to its own Maxwellian distribution $\mu^\e$, and the relative entropy for the  macroscopic quantities:
\begin{equation}\label{}
\cH(f^\e | \mu) = \cH(f^\e| \mu^\e) + \cH(\mu^\e | \mu),
\end{equation}
\begin{equation}\label{}
\cH(\mu^\e| \mu) = \frac12 \int_{\O^n} \rho^\e | u^\e - u|^2 \dx +  \int_{\O^n}  \rho^\e \log( \rho^\e /\rho) \dx.
\end{equation}
So, if $\cH(f^\e | \mu) \to 0$, then also $\cH(\mu^\e | \mu) \to 0$, and it is easy to show that all the macro-limits 
\begin{equation}\label{e:macrolimit}
\begin{split}
\rho^\e & \to \rho,\\
\rho^\e u^\e & \to \rho u, \\
\rho^\e |u^\e|^2 & \to \rho |u|^2.
\end{split}
\end{equation}
hold in $L^1(\O^n)$, see \cite{KMT2015}.

\begin{proposition}\label{p:limit}
Let $(u,\rho)$ be a smooth non-vacuous solution to \eqref{e:macro} on a time interval $[0,T)$ and let $\phi$ be a global kernel, $\min_{\O^n} \phi >0$. Suppose that initial distributions $f_0^\e$ converge to $\mu_0$ in the sense of entropies as $\e \to 0$:
\[
\cH(f^\e_0 | \mu_0) \to 0, 
\]
then for any $t \in [0,T)$,
\[
\cH(f^\e | \mu) \to 0.
\]
\end{proposition}

\begin{remark} The only purpose of the assumption on the communication kernel here is to ensure that the mollified densities all enjoy a common lower bound on the domain in question:
\begin{equation}\label{e:rmin}
\rho^\e_\phi(x,t) > \rmin, \qquad (x,t) \in \O^n \times [0,T).
\end{equation}
Any family of solutions satisfying \eqref{e:rmin} would fit into the framework of the proof of \prop{p:limit} and the convergence result for such a family  would still hold.
\end{remark}

\begin{proof}

We begin by breaking down the relative entropy into kinetic and macroscopic parts:
\begin{equation}\label{}
\cH(f^\e | \mu) = \cH_\e + \cG_\e .
\end{equation}
The kinetic component
\[
\cH_\e = \int_{\domain} \left( f^\e \log f^\e + \frac12 |v|^2 f^\e \right)\dv \dx +  \frac{n M}{2} \log(2\pi)
\]
is exactly the same entropy relative to the basic Maxwellian  \eqref{e:Max0} we considered in the previous section. The macroscopic component is given by
\[
\cG_\e = \int_{\O^n} \left( \frac12 \rho_\e |u|^2 - \rho_\e u_\e \cdot u - \rho_\e \log \rho \right) \dx.
\]
We now state the energy bounds for each component.
\begin{lemma}\label{l:elawepsilon}
We have the following entropy laws:
\begin{align}
\ddt \cH_\e & \leq n M; \label{e:HM} \\ 
\ddt \cH_\e  &= - \frac{1}{\e} \cI_\e + \frac{\e}{4} \int_{\domain} |v -u^\e|^2 f^\e \dx \dv   +\cE_\phi^\e - \cE^\e, \label{e:Hident}
\end{align}
where 
\[
\cI_\e = \int_{\domain}\frac{|\n_v f^\e+\left(1+ \frac{\e}{2}\right)   (v -u^\e) f^\e|^2}{f^\e}  \dx \dv .
\]
\end{lemma}
\begin{proof}
Differentiating,
\begin{equation}\label{e:He}
\begin{split}
\ddt \cH_\e  = &- \frac{1}{\e} \int_{\domain}\left[ \frac{|\n_v f^\e|^2}{f^\e} + 2 \n_v f^\e \cdot (v -u^\e) + |v -u^\e|^2 f^\e \right]\dx \dv \\
& - \int_{\domain} [\n_v f^\e \cdot (v - \ufilt^\e)  + v \cdot (v - \ufilt^\e) f^\e ] \dx \dv.
\end{split}
\end{equation}
To prove \eqref{e:HM} we simply dismiss the information term, and rewrite the filtered term as follows
\[
- \int_{\domain} [\n_v f^\e \cdot (v - \ufilt^\e)  + v \cdot (v - \ufilt^\e) f^\e ] \dx \dv = nM - E^\e + \cE^\e_\phi \leq nM,
\]
where the latter is due to \eqref{e:enhier}. This proves  \eqref{e:HM}.

To show \eqref{e:Hident} we treat the filtered term somewhat differently:
\begin{equation*}\label{}
\begin{split}
\int_{\domain} & [\n_v f^\e \cdot (v - \ufilt^\e)  + v \cdot (v - \ufilt^\e) f^\e ] \dx \dv  = \int_{\domain} \n_v f^\e \cdot v  \dx \dv  +\int_{\domain}  |v|^2 f^\e  \dx \dv - \cE^\e_\phi  \\
& = \int_{\domain} \n_v f^\e \cdot (v - u^\e)   \dx \dv  +  \int_{\domain}  |v - u^\e |^2 f^\e  \dx \dv  + \cE^\e - \cE^\e_\phi.
\end{split}
\end{equation*}
Coming back to the main equation \eqref{e:He} we obtain
\begin{equation*}\label{}
\begin{split}
\ddt \cH_\e  = &- \frac{1}{\e} \int_{\domain}\frac{1}{f^\e} \left[ |\n_v f^\e|^2 + 2\left(1+ \frac{\e}{2}\right) \n_v f^\e \cdot (v -u^\e) f^\e + \left(1+ \frac{\e}{2}\right)^2 |v -u^\e|^2 (f^\e)^2 \right]\dx \dv \\
& + \frac{\e}{4}  \int_{\domain} |v -u^\e|^2 f^\e \dx \dv   +\cE_\phi^\e - \cE^\e,
\end{split}
\end{equation*}
as desired.
\end{proof}

The main consequence of \eqref{e:HM}  is that the entropy $\cH^\e$ remains bounded on the time interval $[0,T)$ \emph{uniformly} in $\e$. This in turn implies uniform bound on the total energy $E^\e$ by way of the argument presented in \sect{s:ee},
\begin{equation}\label{ }
E^\e \leq  C,
\end{equation}
with $C$ independent of $\e$.

\begin{lemma}\label{}
We have the following inequality
\begin{equation}\label{e:Ge}
\ddt \cG_\e \leq C \cH(f^\e | \mu) + C \sqrt{E^\e  \cI_\e} + C\e E^\e+  \cE^\e - \cE_\phi^\e,
\end{equation}
where $C$ is independent of $\e$.
\end{lemma}
\begin{proof} Let us compute the derivative of each component of $\cG_\e$ (we omit the integral signs on the right hand side for short):
\[
\begin{split}
\ddt \frac12 \int_{\O^n} \rho^\e |u|^2 \dx & =  \rho^\e (u^\e - u) \cdot \n u \cdot u - (\rho^\e - \rho) u \cdot \n \log \rho - u \cdot \n \rho + \rho^\e u (\ufilt - u) \\
\ddt \int_{\O^n} \rho^\e u^\e \cdot u  \dx& = \rho^\e (u^\e - u) \cdot \n u \cdot u^\e + \rho^\e \n \cdot u - \rho^\e u^\e \cdot \n \log \rho - \n u: \cR_\e \\
&+  \rho^\e u (\ufilt^\e - u^\e) +  \rho^\e u^\e (\ufilt - u) \\
\ddt \int_{\O^n} \rho^\e \log \rho \dx &= \rho^\e (u^\e - u) \cdot \n \log \rho - \rho^\e \n \cdot u.
\end{split}
\]
Thus,
\[
\ddt \cG_\e = \n u: \cR_\e  +  \rho^\e (u^\e - u) \cdot \n u \cdot (u^\e - u)  + A,
\]
where  $A$ is the alignment component,
\[
A =   \rho^\e u (\ufilt - u) - \rho^\e u (\ufilt^\e - u^\e) - \rho^\e u^\e (\ufilt - u).
\]
Given that $u$ is smooth we have
\begin{equation}\label{e:Ge1}
\ddt \cG_\e \leq C \int_{\O^n}|\cR_\e | \dx+ C \int_{\O^n}  \rho^\e |u^\e - u|^2  \dx + A.
\end{equation}

Let us proceed with the  alignment term by rewriting it as follows
\[
A = \rho^\e (u - u^\e) (\ufilt - u) - \rho^\e (u - u^\e) (\ufilt^\e - u^\e) - \rho^\e u^\e (\ufilt^\e - u^\e).
\]
The last term here is given by
\[
- \rho^\e u^\e (\ufilt^\e - u^\e) = \cE^\e - \cE_\phi^\e.
\]
The remaining first two terms combined give
\[
\rho^\e (u - u^\e) (\ufilt - \ufilt^\e )  - \rho^\e |u - u^\e|^2 \leq \frac12 \rho^\e |\ufilt - \ufilt^\e|^2 - \frac12 \rho^\e |u - u^\e|^2.
\]
It remains to estimate the first term:
\[
\rho^\e |\ufilt - \ufilt^\e|^2 \leq \rho^\e_\phi |\uFavre - \uFavre^\e|^2.
\]
 Let us recall that the filtrations here  are performed with respect to their corresponding densities. To reconcile this descrepency we add and subtract the Favre filtration of  $u$ with respect to $\rho^\e$:
\[
\uFavre = \frac{(u\rho)_\phi}{\rho_\phi} - \frac{(u\rho^\e)_\phi}{\rho^\e_\phi} + \frac{(u\rho^\e)_\phi}{\rho^\e_\phi}.
\]
Thus,
\[
\rho^\e |\ufilt - \ufilt^\e|^2 \leq 2 \rho^\e_\phi \left|\frac{(u\rho^\e)_\phi}{\rho^\e_\phi}-\frac{(u^\e\rho^\e)_\phi}{\rho^\e_\phi} \right|^2 + 2 \rho^\e_\phi \left| \frac{(u\rho)_\phi}{\rho_\phi} - \frac{(u\rho^\e)_\phi}{\rho^\e_\phi} \right|^2.
\]
The first term is estimated by the \HI\ treating $\rho^\e(y) \phi(x-y) \dy/ \rho^\e_\phi(x)$ as a probability measure,
\[
\begin{split}
\int_{\O^n} \rho^\e_\phi \left|\frac{(u\rho^\e)_\phi}{\rho^\e_\phi}-\frac{(u^\e\rho^\e)_\phi}{\rho^\e_\phi} \right|^2 \dx & \leq \int_{\O^n} \rho^\e_\phi(x) \frac{ \int_{\O^n} |u(y) - u^\e(y)|^2 \rho^\e(y) \phi(x-y) \dy}{\rho^\e_\phi(x)} \dx = \int_{\O^n} \rho^\e |u - u^\e|^2 \dy\\
& \leq \cH(f^\e|\mu).
\end{split}
\]
The second term can be estimated by 
\[
\int_{\O^n} \rho^\e_\phi \left| \frac{(u\rho)_\phi}{\rho_\phi} - \frac{(u\rho^\e)_\phi}{\rho^\e_\phi} \right|^2 \dx \lesssim \int_{\O^n} \frac{| (u\rho)_\phi (\rho^\e - \rho)_\phi |^2}{\rho_\phi^2 \rho_\phi^\e} \dx + \int_{\O^n} \frac{| (u  (\rho^\e - \rho))_\phi |^2}{\rho_\phi^\e} \dx.
\]
Using the simple pointwise estimates
\[
| (u(\rho^\e - \rho))_\phi(x)|,  | (\rho^\e - \rho)_\phi(x)| \lesssim \|\rho^\e - \rho\|_1,
\]
and the fact that the densities are bounded away from zero on the interval $[0,T)$, \eqref{e:rmin}, we obtain
\[
\int_{\O^n} \rho^\e_\phi \left| \frac{(u\rho)_\phi}{\rho_\phi} - \frac{(u\rho^\e)_\phi}{\rho^\e_\phi} \right|^2 \dx \leq C  \|\rho^\e - \rho\|_1^2 \leq C \cH(\rho^\e | \rho) \leq C\cH(f^\e | \mu).
\]
Combining the above we obtain
\[
A \leq C \cH(f^\e | \mu) +  \cE^\e - \cE_\phi^\e.
\]
Thus, 
\begin{equation}\label{e:Ge2}
\ddt \cG_\e \leq C \int_{\O^n}|\cR_\e | \dx+ C \cH(f^\e | \mu) +  \cE^\e - \cE_\phi^\e.
\end{equation}
It remains to estimate the Reynolds stress. A well-known inequality of \cite{MV2008} establishes such a bound in terms of information and energy. Let us rerun this argument to account for the $\e$-correction. Using that 
\[
\int_{\R^n} u^\e \otimes \n_v f^\e \dv = 0, \qquad \int_{\R^n}  \n_v f^\e \otimes \I \dv= -  \int_{\R^n} f^\e \I \dv,
\]
we write
\[
\begin{split}
\cR_\e & = \int_{\R^n} u^\e \sqrt{f^\e} \otimes ( 2 \n_v \sqrt{f^\e} + (v - u^\e) \sqrt{f^\e} ) + ( 2 \n_v \sqrt{f^\e} + (v - u^\e) \sqrt{f^\e} ) \otimes v \sqrt{f^\e}  \ \dv \\
& =  \int_{\R^n} u^\e \sqrt{f^\e} \otimes ( 2 \n_v \sqrt{f^\e} + (1+\e/2) (v - u^\e) \sqrt{f^\e} ) + ( 2 \n_v \sqrt{f^\e} + (1+\e/2)(v - u^\e) \sqrt{f^\e} ) \otimes v \sqrt{f^\e}  \ \dv \\
& - \frac{\e}{2}  \int_{\R^n} [u^\e \otimes (v - u^\e) + (v - u^\e) \otimes u^\e] f^\e  \dv.\end{split}
\]
Thus,
\[
 \int_{\O^n}|\cR_\e | \dx \leq C \sqrt{E^\e  \cI_\e} + C\e E^\e
\]
and the lemma is proved.
\end{proof}

Combining the kinetic and macroscopic laws \eqref{e:Hident}, \eqref{e:Ge}  
we can see that the residual energy $\cE^\e - \cE_\phi^\e$ cancels out and we obtain
\[
\ddt \cH(f^\e | \mu)  \leq C\cH(f^\e | \mu)   - \frac{1}{\e} \cI_\e  + C \e E^\e + C \sqrt{E^\e  \cI_\e} \leq C\cH(f^\e | \mu)  - \frac{1}{2\e} \cI_\e + C \e E^\e \leq C_1 \cH(f^\e | \mu) +C_2,
\]
where $C_i$'s are independent of $\e$. Since initial entropy $\cH(f^\e_0 | \mu) $ vanishes as $\e \to 0$, the \GL\ finishes the proof.

\end{proof}

\section{Appendix: Well-posedness and continuation}\label{s:wp}
In this section we collect and address all the basic issues of well-posedness and  continuation of classical solutions of \eqref{e:FPA}.

To set the stage let us fix value $\s = 1$ as it plays no role in the analysis. Let us consider first the linear FPA model, 
\begin{equation}\label{e:FPAL}
\p_t f + v\cdot \n_x f =  \D_v f + \n_v((v - u)f ), 
\end{equation}
where $u \in L^\infty_\loc([0,T); C^k(\O^n))$ is a given macroscopic field. The well-posedness of solutions on $[0,T)$ in any class $H^k_s$ for this equation follows by the standard linear theory, see for example \cite{Villani} and references therein. Our main existence result holds in $H^k_s(\domain)$, for $k,s> N$, where $N$ is large and dependent only on $n$.
\begin{theorem}\label{t:lwp}
Suppose $f_0\in H^k_s(\domain)$,  is such that $\rmin = \inf (\rho_0)_\phi >0$. Then there exists a unique local solution to \eqref{e:FPA} on a time interval $[0,T)$, where $T >0$ depends only on $E_0$ and $\rmin$,  in the same class
\begin{equation}\label{e:solclass}
f\in L^\infty ([0,T); H^k_s ), \quad \inf_{[0,T)\times \O^n} \rho_\phi >0.
\end{equation}
Moreover, if $f\in L^\infty_\loc ([0,T); H^k_s )$ is a given solution such that 
\begin{equation}\label{e:rholow}
\inf_{[0,T)\times \O^n} \rho_\phi >0,
\end{equation}
then $f$ can be extended beyond $T$ in the same class.
\end{theorem}

\begin{proof} The solution will be constructed by an iteration given by $f^0 \equiv f_0$, and 
\begin{equation}\label{e:FPALm}
\begin{split}
\p_t f^{m+1} + v\cdot \n_x f^{m+1} & =  \D_v f^{m+1} + \n_v((v - u^{m})f^{m+1} ), \\
f_0^{m+1} & = f_0.
\end{split}
\end{equation}
where $u^m = u^{m}_{\phi,\rho^m}$.

Let us show that solutions to the above system exist on a common time interval $[0,T)$, where $T$ depends on $E_0$, $\rmin$, and $M$. In fact it suffices to show that on a common time interval the solutions will have a common bound on the energy and density
\begin{equation}\label{e:Er}
E^m(t) \leq 2E_0, \qquad \rho^m_\phi(t) \geq \frac12 \rmin.
\end{equation}
For $m=0$ this is obviously true with $T_0 = \infty$. Suppose \eqref{e:Er} holds for $t<T_m$. Then, since  $u^m \in L^\infty_\loc([0,T_m); C^k(\O^n))$ the solution $f^{m+1}$ exists as least on the same time interval. Estimating pointwise,
\[
| \uFavre^m | \leq 2\rmin^{-1} \|\phi\|_\infty \int_{\O^n} |u^m \rho^m| \dx \leq 2\rmin^{-1} \|\phi\|_\infty M^{1/2} \sqrt{E^m} \leq 2^{3/2} \rmin^{-1} \|\phi\|_\infty M^{1/2} E_0^{1/2} = C \rmin^{-1} M^{1/2} E_0^{1/2},
\]
where $C$ captures all the dependence on the parameters of the system only. Calculating the energy for $f^{m+1}$ from \eqref{e:FPALm} we obtain
\begin{equation}\label{e:Emaux}
\begin{split}
\frac12 \ddt E^{m+1}  & = nM - E^{m+1} + \int_{\O^n} (u^{m+1} \rho^{m+1})_\phi \uFavre^m \dx\\
& \leq nM - E^{m+1} + C \rmin^{-1} M^{1/2} E_0^{1/2}\int_{\O^n} |u^{m+1} \rho^{m+1}| \dx  \\
&\leq nM - E^{m+1} + C \rmin^{-1} M E_0^{1/2}  \sqrt{E^{m+1}} \leq nM - \frac12 E^{m+1} + C \rmin^{-2} M^2 E_0.  
\end{split}
\end{equation}
Consequently,
\[
E^{m+1}(t) \leq E_0 e^{-t} + (2nM + C \rmin^{-2} M^2 E_0) (1 - e^{-t}) \leq 2 E_0,
\]
provided $t <  \frac{c E_0}{2nM + C \rmin^{-2} M^2 E_0} := T_{m+1}$. Hence for $t < T_{m} \wedge T_{m+1}$ we have
\begin{equation*}\label{}
\begin{split}
\p_t \rho^{m+1}_\phi & = \n_x \cdot (u^{m+1} \rho^{m+1})_\phi = (u^{m+1} \rho^{m+1})_{\n \phi}  \geq - \| \n\phi\|_\infty \int_{\O^n} |u^{m+1} \rho^{m+1}| \dx \\
&  \geq  - C M^{1/2} \sqrt{E^{m+1}} \geq - C \sqrt{M E_0}.
\end{split}
\end{equation*}
So, pointwise,
\begin{equation}\label{e:rlowm}
\rho^{m+1}_\phi (t) \geq \rmin - t C \sqrt{M E_0} \geq \frac12 \rmin,
\end{equation}
provided $t< C\rmin (M E_0)^{-1/2}$. Resetting 
\[
T_{m+1} = \min\left\{ T_m, C\rmin (M E_0)^{-1/2},  \frac{c E_0}{2nM + C \rmin^{-2} M^2 E_0} \right\},
\]
we can see that the new restriction on time is independent of $m$. Since initially $T_0 = \infty$ the induction proves that the solutions will exist on the common time interval $[0,T)$ with
\[
T = \min\left\{C\rmin (M E_0)^{-1/2},  \frac{c E_0}{2nM + C \rmin^{-2} M^2 E_0} \right\}.
\]

We have constructed a sequence of solutions $f^m$ satisfying \eqref{e:Er} on a common interval $[0,T)$. This implies uniform bounds on the family in $H^k_s$. Indeed, all the norms $\| u^m(t) \|_{C^k}$ for $t<T$ depend only on the bounds \eqref{e:Er}, while the standard energy estimates provide an exponential bound on $\|f^m\|_{H^k_s}$ only in terms of  $\| u^m \|_{L^\infty C^k}= C(E_0,\rmin)$,
\begin{equation}\label{e:fmH}
\|f^m(t) \|_{H^k_s} \leq \|f^m_0\|_{H^k_s} e^{C(E_0,\rmin) t}.
\end{equation}
Next, let us estimate the time derivative in $L^2$. We have
\[
\| \p_t f^m \|_2 \leq \| v \cdot \n_x f^m \|_2 + \| \D_v f^m \|_2 + n \|f^m\|_2 + \| (v - u^{m-1}) \cdot \n_v f^m\|_2 \leq  C(E_0,\rmin) \|f^m\|_{H^k_s},  
\]
which according to \eqref{e:fmH} is uniformly bounded on $[0,T)$. Thus, 
\[
f^m \in L^\infty([0,T); H^k_s) \cap \Lip([0,T); L^2),
\]
uniformly. In view of the fact that $H^k_s \ss H^{k'}_{s'}$, $k'<k$, $s'<s$, compactly, and of course $H^{k'}_{s'} \ss L^2$,  the Aubin-Lions Lemma implies compactness of the family in any $C([0,T);  H^{k'}_{s'})$. Passing to a subsequence we obtain a solution $f^m \to f$. It is easy to show that in this case $u^m \to \ufilt$ in any $C^l$, $l\in \N$, which is more than necessary to conclude that $f$ solves \eqref{e:FPA}. By weak compactness we also obtain membership in the top space $f\in L^\infty([0,T); H^k_s)$.

\def \tf {\tilde{f}}
\def \tu {\tilde{u}}
\def \tuF {\tilde{u}_{\mathrm{F}}}
\def \tr {\tilde{\rho}}

Let us have two solutions $f$ and $\tf$ in class \eqref{e:solclass} starting from the same initial condition $f_0$. Denote $g = f - \tf$. We will estimate evolution of this difference in the weighted class $L^2_s = H^0_s$. Let us take the difference 
\[
\p_t g + v \cdot \n_x g = \D_v g + \n_v ( v g - \uf g + (\tuf - \uf) \tf).
\]
Testing with $\ave{v}^s g$ we obtain
\begin{equation*}\label{}
\begin{split}
\ddt \|g\|_{L^2_s}^2 & \leq  - \|\n_v g\|_{L^2_s}^2 + \int_\domain \ave{v}^{s-1} |\n_v g| |g| \dv\dx  \\
&- \int_\domain  \n_v  g ( v g - \uf g + (\tuf - \uf) \tf) \ave{v}^s \dv\dx\\
& + \int_\domain  | g| | v g - \uf g + (\tuf - \uf) \tf| \ave{v}^{s-1} \dv\dx
\end{split}
\end{equation*}
We have
\[
 \int_\domain \ave{v}^{s-1} |\n_v g| |g| \dv\dx \leq  \|\n_v g\|_{L^2_s}  \|g\|_{L^2_{s-2}} \leq \frac14 \|\n_v g\|_{L^2_s}^2 +  c \|g\|_{L^2_{s}}^2.
 \]
 By a similar estimate and using that $\uf$ is bounded, we obtain
 \begin{equation*}\label{}
\begin{split}
 \int_\domain  \n_v  g ( v g - \uf g) \dv \dx & \leq  \frac14 \|\n_v g\|_{L^2_s}^2 +  c \|g\|_{L^2_{s}}^2, \\
  \int_\domain  | g| | v g - \uf g | \ave{v}^{s-1} \dv\dx & \leq c \|g\|_{L^2_{s}}^2.
\end{split}
\end{equation*}
Next,
\begin{equation*}\label{}
\begin{split}
\int_\domain | \n_v  g| | (\tuf - \uf) \tf| \ave{v}^s \dv\dx & \leq \frac14 \|\n_v g\|_{L^2_s}^2 + \| \tuf - \uf \|_\infty^2 \| \tf\|^2_{L^2_s} \\
 \int_\domain  | g| |(\tuf - \uf) \tf| \ave{v}^{s-1} \dv\dx & \leq \| g\|_{L^2_s}^2 +  \| \tuf - \uf \|_\infty^2 \| \tf\|^2_{L^2_s}.
\end{split}
\end{equation*}
Recall that $\| \tf\|^2_{L^2_s} \leq C$ on the interval of existence.  So, adding the above inequalities we obtain
\[
\ddt \|g\|_{L^2_s}^2 \lesssim C  \| g\|_{L^2_s}^2  +  \| \tuf - \uf \|_\infty^2.
\]
Finally,
\[
 \| \tuf - \uf \|_\infty^2 \leq \|\phi\|_\infty  \int_{\O^n} | \uF (y) - \tuF(y) |^2 \dy,
\]
and using the lower bound on the density and the fact that $(u\rho)_\phi, (\tu \tilde{\rho})_\phi$ remain uniformly bounded, we obtain
\begin{equation*}\label{}
\begin{split}
 \| \tuf - \uf \|_\infty^2 & \lesssim \int_{\O^n} | (u\rho -\tu \tilde{\rho})_\phi |^2 \dy +  \int_{\O^n} | (\rho -\tilde{\rho})_\phi |^2 \dy \\
  & \lesssim \int_{\O^n} |u\rho -\tu \tilde{\rho}|^2 \dy +  \int_{\O^n} | \rho -\tilde{\rho} |^2 \dy.
\end{split}
\end{equation*}
The differences can be estimated as follows (keeping in mind that $s>n+2$),
\begin{equation*}\label{}
\begin{split}
\int_{\O^n} |u\rho -\tu \tilde{\rho}|^2 \dy & = \int_{\O^n} \left( \int_{\R^n} |v| |g| \dv \right)^2 \dy =  \int_{\O^n} \left( \int_{\R^n} \ave{v}^{s/2} |g| \frac{\dv}{\ave{v}^{s/2 - 1 }}  \right)^2 \dy \\
& \leq C \int_{\O^n} \int_{\R^n} \ave{v}^{s} |g|^2 \dv  \dy = C \| g\|_{L^2_s}^2.
\end{split}
\end{equation*}
Similarly,
\[
  \int_{\O^n} | \rho -\tilde{\rho} |^2 \dy  \leq C \| g\|_{L^2_s}^2.
\]
We thus obtain
\[
\ddt \|g\|_{L^2_s}^2 \leq C  \| g\|_{L^2_s}^2 .
\]
Since initially $g = 0$, the result follows.

The continuation criterion follows readily from the above. Notice that according to \eqref{e:Eunif}  the energy will remain uniformly bounded by on the interval of existence: $E(t) \leq C_1 \cH_0 +C_2$. Together with the assumption \eqref{e:rholow} it implies that the solution will exist on a finite time-span $T_0$ which depends only on $E_0,M,T$, starting from any time $<T$. By uniqueness the extended solution will coincide with the original one on the overlap.
\end{proof}


\end{document}